\documentclass{article}
\usepackage{latexsym}
\usepackage{mathrsfs}
\usepackage{indentfirst}
\usepackage{amsmath,amsthm, amssymb}
\usepackage{graphicx}
\usepackage{epsfig}
\usepackage{verbatim}
\usepackage{color}

\usepackage{caption}

\usepackage{enumitem}

\usepackage{epstopdf}
\usepackage[hidelinks]{hyperref}

\usepackage{thmtools}

\newtheorem{theorem}{Theorem}[section]
\newtheorem{lemma}{Lemma}

\newtheorem{problem}{Problem}

\newtheorem{proposition}{Proposition}

\newtheorem{claim}{Claim}
\allowdisplaybreaks[2]
\newcommand{\BGL}{\operatorname{BGL}}
\newcommand{\LSD}{\operatorname{LSD}}
\newcommand{\caseheading}[1]{\par\medskip\noindent\textbf{#1}\quad}

\begin{document}
\title{A Bang-Jensen-Gutin-Li type condition for pancyclicity in digraphs
}
\author{ Tianyue Li$^{1}$,
Zan-Bo Zhang$^{1,2}$\thanks{Corresponding author. Email address: eltonzhang2001@gmail.com, zanbozhang@gdufe.edu.cn. 
Supported by Guangdong Basic and Applied Basic Research Foundation (2024A1515012286),
and Guangdong Provincial Department of Education Research Project - Characteristic Innovation Project (2024KTSCX204).}
\\ \small $^1$ School of Statistics and Data Science, Guangdong University of
\\ \small  Finance \& Economics, Guangzhou 510320, China
\\ \small $^2$ Institute of Artificial Intelligence and Deep Learning, Guangdong University
\\ \small of Finance \& Economics, Guangzhou 510320, China
}
\date{}
\maketitle

\begin{abstract} 
\noindent

Bang-Jensen-Gutin-Li type conditions are degree conditions for the existence of 
hamiltonian cycles in digraphs,
which impose degree restrictions on nonadjacent vertices which have a common in-neighbor
or a common out-neighbor.
They can be viewed as directed analogues of both Fan‑type conditions and 
heavy subgraph conditions in undirected graphs, 
and they also generalize the classes of locally semicomplete digraphs, locally in‑semicomplete digraphs,
and locally out‑semicomplete digraphs.
In this paper, we prove what is, to our knowledge, the first Bang-Jensen-Gutin-Li 
type condition for pancyclicity in digraphs.
More precisely, let $D$ be a strong digraph with order $n\ge 3$.
We show that, if for every $u\in V(D)$ belonging to a
dominated nonadjacent pair or a dominating nonadjacent pair, we have $d(u)\ge n$,
then $D$ is pancyclic, unless $D$ belongs to one of two well-characterized
classes of exceptional digraphs.

$\newline$\noindent\textbf{Key words}: hamiltonicity; Bang-Jensen-Gutin-Li condition; pancyclicity; digraph;
\end{abstract}

\section{Introduction, terminology and notation} \label{sec:Intro}

Throughout this paper, we consider only strict digraphs, i.e., digraphs without loops or parallel arcs.
A notation summary will be given at the end of this section. 
Any unclear notation may be found there or in reference \cite{BG2007}.

The existence conditions for hamiltonian cycles in digraphs have long been a popular research topic.
Classical conditions include following ones: A digraph $D$ on $n\geq 2$ vertices is hamiltonian if
\begin{enumerate}
	\item (Ghouila-Houri \cite{Gho1960}) $d(u)\geq n$ for every $u\in V(D)$.
	\item (Woodall \cite{Woo1972}) $d^+(u)+d^-(v) \geq n$ for 
	all pairs of vertices $u,v \in V(D)$ such that there is no arc from $u$ to $v$.
	\item (Meyniel \cite{Mey1973}) $d(u)+d(v) \geq 2n-1$ for 
	all pairs of nonadjacetn vertices $u,v \in V(D)$.
\end{enumerate}

In their seminal paper \cite{BGL1996}, Bang‑Jensen, Gutin and Li introduced a new type of condition
that imposes degree restrictions on the following vertex pairs in a digraph.
For two nonadjacent vertices $u$ and $v$,
if they have a common in-neighbor or a common out-neighbor, 
they are called a dominating nonadjacent pair or a dominated nonadjacent pair, respectively.
For short, we call this type of condition a BGL‑type condition.
Their main results are as follows.

\begin{theorem}\label{thm:BGL-hamiltonian1}
Let $D$ be a strong digraph of order $n$. Suppose that, for every dominated nonadjacent pair
$\{x,y\}$, either $d(x) \geq n$ and $d(y) \geq n - 1$ or
$d(x) \geq n - 1$ and $d(y) \geq n$, then $D$ is hamiltonian.
\end{theorem}

\begin{theorem}\label{thm:BGL-hamiltonian2}
	Let $D$ be a strong digraph of order $n$. Suppose that, for every pair of dominating
	nonadjacent and every pair of dominated nonadjacent vertices $\{x,y\}$,
	$\min\{d^+(x)+d^-(y), d^-(x)+d^+(y) \}\geq n$, then $D$ is hamiltonian.
\end{theorem}

Theorem~\ref{thm:BGL-hamiltonian1} strengthens the condition of Ghouila-Houri, and
Theorem~\ref{thm:BGL-hamiltonian2} strengthens the condition of Woodall,
while a similar strengthening of Meyniel condition conjectured in the same paper
remains open to this day.

The conditions involving dominated nonadjacent pair and dominating nonadjacent pairs are closely related to
many important concepts and ideas in the study of cycle structures. We name a few below.

\begin{enumerate}
\item They can be viewed as the extension of Fan-type conditions (\cite{Fan1984}) from
undirected graphs to digraphs.
The widely used Fan-type conditions strengthen various degree conditions 
by restricting degree lower bounds on vertex pairs at distance two only.
Since the two vertices in any dominated or dominating pair have distance two in the underlying graph, 
BGL‑type conditions serve as the natural directed counterpart.

\item 
They generalize three classes of digraphs: locally in-semicomplete digraphs (LISDs), 
locally out-semicomplete digraphs (LOSDs), and locally semicomplete digraphs (LSDs). 
An \textit{LISD} is a digraph in which the in-neighborhood of every vertex induces a semicomplete digraph; 
an \textit{LOSD} is defined analogously with respect to out-neighborhoods. 
An \textit{LSD} is both locally in-semicomplete and locally out-semicomplete. 
Consequently, LISDs contain no nonadjacent dominating pairs, LOSDs contain no nonadjacent dominated pairs, 
and LSDs contain neither.
Therefore, any BGL-type condition restricting the degrees of dominated
nonadjacent pairs must contain all LOSDs and LSDs.
Similarly, any BGL-type condition restricting the degrees of dominating 
nonadjacent pairs must contain all LISDs and LSDs.

From the above discussion we see that Theorem~\ref{thm:BGL-hamiltonian1} 
immediately gives Hamiltonicity for strong LOSDs and strong LSDs. 
Since reversing arcs exchanges dominated and dominating pairs but preserves hamiltonicity and total degrees, 
the theorem also covers strong LISDs. The hamiltonicity of these classes was previously proved in 
\cite{JBJ1990,BHP1993}. 

\item They can be viewed as digraph analogue of heavy subgraph conditions,
where a typical heavy subgraph condition requires that
if $H$ is an induced copy of a fixed graph $F$,
then the degrees or degree sum of certain vertices in $H$ satisfy prescribed
lower bounds.
BGL‑type conditions are precisely the directed analogues in which 
$F$ is taken as one of the subdigraphs obtained by orienting the edges in a path of length two
in the opposite direction.
\end{enumerate}

Bondy's meta-conjecture states that any nontrivial condition for hamiltonicity also implies pancyclicity
up to a small class of exceptional graphs (\cite{Bondy1973}).
In view of this idea, various Fan-type conditions as well as heavy subgraph conditions
have been proved to imply pancyclicity. 
For digraphs, the afore-mentioned Ghouila-Houri condition, Woodall condition and Meyniel condition
have all been extended to pancyclicity as well (for example, see \cite{HT1976}).
The pancyclicity of LISDs, LOSDs and LSDs has also been well-studied.
In particular, the class of non-pancyclic LSDs, which will be an exceptional class
of our main result, has a nice characterization as below.

A digraph on $n$ vertices is called a \emph{round digraph} if we can label its
vertices $v_0$, $v_1, \dots , v_{n-1}$ such that for each $i$,
$N^+(v_i) = \{v_{i+1}, \dots , v_{i+d^+(v_i)} \}$ and
$N^-(v_i) = \{v_{i-d^-(v_i)}, \dots , v_{i-1}\}$, where the subscripts are taken modulo $n$.
An LSD is called a \textit{local tournament} if it contains no opposite arcs.
An LSD $D$ is said to be \emph{round decomposable}, if it can be constructed from
a local tournament $R$ with $r$ vertecies $\{v_0, v_1, \dots, v_{r-1}\}$
by replacing every vertex $v_i$ with a strong semicomplete digraph $S_i$,
and every arc $v_iv_j$ with all arcs from $V(S_i)$ to $V(S_j)$.
Such a construction is called the \emph{round decomposition} of $D$ and denoted by 
$D = R[S_0, S_1, \dots , S_{r-1}]$. 
The \textit{girth} of a digraph $D$, denoted by $g(D)$, is the length of the shortest cycle in $D$.
The non-pancyclic LSDs are precisely a subclass of the round-decomposable LSDs described as follows.

\begin{theorem}[Guo~\cite{Guo1995}]\label{thm:nonpancyclicLSD}
A strong LSD $D$ is pancyclic if and only if it is not of the form\linebreak[3]
$D = R[S_0, S_1, \dots , S_{r-1}]$, where $R$ is a round local tournament with
$g(R) > \max\{2, \allowbreak |S_0|, \allowbreak |S_1|, \allowbreak \dots , \allowbreak |S_{r-1}|\} +1$.
\end{theorem}

In contrast to the above results, no BGL-type 
condition for pancyclicity appears to exist in the literature.
The problem of generalizing such conditions to pancyclicity was first explicitly 
stated by Darbinyan and Karapetyan~\cite{DK2017}, 
who obtained the following partial result towards 
generalizing of Theorem~\ref{thm:BGL-hamiltonian1} to pancyclicity. 

\begin{theorem}\label{thm:DK2017}
Let $D$ be a strongly connected digraph
of order $n \geq 3$ with minimum semi-degree at least two. 
Suppose that $d(x) \geq n$ and $d(y) \geq n-1$ or $d(x) \geq n-1$
 and $d(y) \geq n$ for every dominated nonadjacent pair
$\{x, y\}$. Then either $D$ contains a cycle of length $n-1$ or $n$ is even and $D$ belongs to
the class of balanced complete bipartite digraphs or to the class of digraphs obtained
by removing one arc from a balanced complete bipartite digraph.
\end{theorem}

A potential obstacle to extending Theorem~\ref{thm:BGL-hamiltonian1} to pancyclicity lies in the following fact.
While its condition imposes degree restrictions only on dominated nonadjacent pairs,
the theorem applies to LISDs. Therefore, the exceptional class 
must include non-pancyclic LISDs. Unlike the case of LSDs, however,
a complete characterization of non-pancyclic LISDs seems to
be challenging and has not been achieved.

Therefore, it seems reasonable to seek conditions that impose degree restrictions on both 
dominating and dominated nonadjacent pairs. In this paper, we present what appears to be 
the first result in this direction, as stated below. This result was conjectured earlier in \cite{ZWHL2026}.
Let $\mathcal{D}_L$ be the class of non-pancyclic LSDs as described in Theorem~\ref{thm:nonpancyclicLSD},
and let $\mathcal{D}_B$ be the class of balanced complete bipartite digraphs with at least $4$ vertices.

\begin{theorem}\label{thm:main}
Let $D$ be a digraph with order $n$ in which for every $u\in V(D)$ belong to a dominated nonadjacent pair 
or a dominating nonadjacent pair, $d(u)\geq n$, then $D$ is pancyclic, unless $D$ belongs to $\mathcal{D}_L$
or $\mathcal{D}_B$.
\end{theorem}

We sketch our proof here.
The first key step is to establish the existence of an $(n-1)$-cycle under the theorem’s condition; 
this is carried in the next section. Specifically, a digraph $D$ satisfying the condition 
either contains an $(n-1)$-cycle, belongs to $\mathcal{D}_B$, or is itself a directed cycle,
which is a member of $\mathcal{D}_L$. This shows that $D$ contains an $(n-1)$-cycle unless
it belongs to the exceptional classes.

By a result of H\"{a}ggkvist and Thomassen (Lemma \ref{lem:outside-heavy}, Section \ref{sec:main-proof}), 
such an cycle $Q=u_0u_1 \dots u_{n-2}u_0$, together with the unique outside vertex $v$ of degree at least $n$,
guarantee pancyclicity. Hence if $d(v)\ge n$, we are done.
Suppose instead that $d(v)<n$.
If $v$ is adjacent in both direction to a certain $u_i$ on $Q$ with $d(u_i)\ge n$,
then some insertion techniques (as in the proofs of Propositions \ref{prop:nonheavy-segment} and \ref{prop:opposite-heavy})
yield another $(n-1)$-cycle excluding $u_i$,
and pancyclicity again follows from the result of H\"{a}ggkvist and Thomassen. 

The main proof proceeds by a carefully designed recursion. 
If either of the above conditions holds --- that is,
$d(v)\ge n$ or $d(v)<n$ but $v$ is mutually adjacent to some $u_i$ with $d(u_i)\ge n$ --- then $D$ is pancyclic.
In the remaining case, we show that $D-v$ still satisfies the theorem’s condition, so induction applies.
In the course of the induction, 
we construct a sequence of digraphs terminating at either a directed $3$-cycle,
a pancyclic digraph, or a digraph in $\mathcal{D}_L \cup \mathcal{D}_B$.
In the first two cases we conclude that the original graph $D$ is pancyclic.
In the last case, we show that $D$ is either pancyclic, or $D$ is a non-pancyclic LSD, that is,
a member of $\mathcal{D}_L$, which completes the proof.

We have provided a proof on the existence of an $(n-1)$-cycle in \cite{ZWHL2026} using the multi-insertion technique. 
The proof we present in Section \ref{sec:n-1} is a shorter alternative proof based on Theorem~\ref{thm:DK2017}.
The final proof of Theorem \ref{thm:main} will then be given in Section~\ref{sec:main-proof}.
The paper is concluded with tow open problems in Section~\ref{sec:OpenProblem}.
$\newline$

\noindent\textit{Notation.}
Let $D$ be a digraph. We write $V(D)$ and $A(D)$ for the vertex set and arc set of $D$. 
We use $n$ or $|D|$ to denote $|V(D)|$. 
For $v\in V(D)$, the symbols $N_D^+(v)$ and $N_D^-(v)$ denote the out-neighborhood and in-neighborhood of $v$,
while $d_D^+(v)$, $d_D^-(v)$, and $d_D(v)=d_D^+(v)+d_D^-(v)$ denote its out-degree, in-degree, and total degree, 
where degree subscripts are omitted when the ambient digraph is clear. 
Furthermore, $\delta^+(D)$ and $\delta^-(D)$ denote the minimum out-degree and minimum in-degree, 
$\delta^0(D)=\min\{\delta^+(D),\delta^-(D)\}$ is the minimum semi-degree. 
We write $x\to y$ to mean that $xy\in A(D)$,  $x\nrightarrow y$ to mean that $xy\notin A(D)$
and $x\leftrightarrow y$ means that both $x\to y$ and $y\to x$ are arcs. 
The same notation $x\to y$ is also used to denote the arc $xy$ itself, when the context is clear.
For vertices or vertex sets $X$ and $Y$, the notation $X\to Y$ means that every vertex of $X$ sends an arc to every vertex of $Y$, 
with a singleton identified with its unique vertex. 
If $X\subseteq V(D)$, then $D[X]$ is the subdigraph induced by $X$, $D-X=D[V(D)\setminus X]$, and $D-v$ abbreviates $D-\{v\}$;
the notation $D-xy$ means that the arc $xy$ is deleted, while $F+v$ denotes the induced subdigraph on $V(F)\cup\{v\}$ in the ambient digraph. The symbol $C_k$ is reserved for a directed cycle of length $k$ viewed as a standard digraph, whereas a principal cycle selected within a proof is denoted by $Q$, with subscripts or primes when needed; any additional auxiliary cycles or paths are named locally. If $Q=x_0x_1\cdots x_{k-1}x_0$ is a directed cycle, its vertex subscripts are read modulo $k$; more generally, for a directed path or cycle $R$, the notation $R[x,y]$ denotes the directed segment from $x$ to $y$ along $R$. The symbols $\mathcal{D}_L$, $\mathcal{D}_B$, $\BGL(n)$, and $\LSD$ retain the meanings specified in the text.

\section{The existence of $\boldsymbol{(n-1)}$-cycle}\label{sec:n-1}

Throughout this section and the next, the following shorthand will be useful.
A dominating nonadjacent pair or a dominated nonadjacent pair is called a \emph{bad pair}.
Thus a digraph is an LSD if and only if it contains no bad pairs.
A vertex $x$ in a digraph $D$ with $d(x)\geq |D|=n$ is called a \emph{heavy vertex}. 
A strongly connected digraph $D$ of order $n$ is said to satisfy the \emph{$\BGL(n)$ condition}
if every endvertex of every bad pair in $D$ is heavy, 
i.e., if the hypothesis of Theorem \ref{thm:main} holds in $D$.
Let $P=u_0u_1\cdots u_s$ and $Q=v_0v_1\cdots v_t$ be vertex-disjoint directed paths in a digraph $D$. The path $P$ is said to be \emph{insertable into $Q$} if there is an index $i$, $0\leq i\leq t-1$, such that $v_i\to u_0$ and $u_s\to v_{i+1}$. In this situation, the arc $v_iv_{i+1}$ is called an \emph{insertion partner} of $P$, and replacing this arc by $v_iPv_{i+1}$ gives the $(v_0,v_t)$-path
\[
Q[v_0,v_i]P Q[v_{i+1},v_t]
\]
with vertex set $V(P)\cup V(Q)$. The same terminology applies when $Q=v_0v_1\cdots v_tv_0$ is a directed cycle, with subscripts read modulo $t+1$.

With the terminology in place, the next theorem supplies the $(n-1)$-cycle used in the recursive part of the proof.

\begin{theorem}\label{thm:n-minus-one}
		Let $D$ be a strongly connected digraph of order $n\geq4$. Suppose that, for every vertex $u \in V(D)$ belonging to a dominating nonadjacent pair or a dominated nonadjacent pair, we have $d(u) \geq n$. Then $D$ contains a cycle of length $n-1$, unless $n$ is even and $D \in \mathcal{D}_B$, or $D = C_n$ (which belongs to $\mathcal{D}_L$).
	\end{theorem}

\begin{proof}
	The hypothesis of the theorem is stronger than the degree condition in
	Theorem~\ref{thm:BGL-hamiltonian1}, hence $D$ is Hamiltonian. Suppose from now
	on that $D$ contains no cycle of length $n-1$; we shall prove that $D$ must be
	one of the two exceptional digraphs in the statement.

	Since $D$ is strong, every vertex has positive in-degree and positive
	out-degree, and consequently the semi-degree $\delta^0(D)\geq 1$. We distinguish two cases.

	\caseheading{Case 1.}$\delta^0(D)\geq 2$.
	
	For every dominated nonadjacent pair $\{x,y\}$, the hypothesis gives
	$d(x)\geq n$ and $d(y)\geq n$. In particular,
	\[
		d(x)+d(y)\geq 2n-1
		\quad\text{and}\quad
		\min\{d(x),d(y)\}\geq n-1.
	\]
	Theorem~\ref{thm:DK2017} now implies that $n$ is even and that $D$ is either
	a balanced complete bipartite digraph or a digraph obtained from a balanced
	complete bipartite digraph by deleting one arc. It remains only to rule out
	the latter possibility.

	Suppose that $H$ is a balanced complete bipartite digraph with partite sets
	$X$ and $Y$, where $|X|=|Y|=m$, and that
	\[
		D=H-xy,
		\qquad x\in X,\quad y\in Y,\quad n=2m.
	\]
	Since $n\geq4$, we have $m\geq2$; choose
	$x'\in X\setminus\{x\}$. The arcs $y\to x$ and $y\to x'$ are both present
	in $D$, whereas $x$ and $x'$ are nonadjacent. Thus $\{x,x'\}$ is a
	dominated nonadjacent pair. On the other hand, deleting the arc $x\to y$
	from $H$ gives
	\[
		d_D(x)=2m-1=n-1,
	\]
	contrary to the hypothesis of the theorem. Therefore the one-arc deletion
	cannot occur, and hence $D\in\mathcal{D}_B$.

	\caseheading{Case 2.}$\delta^0(D)=1$.
	
	There is a vertex whose in-degree or out-degree is one. Reversing all arcs
	interchanges dominated nonadjacent pairs and dominating nonadjacent pairs,
	while preserving the hypothesis and all conclusions of the theorem.
	Therefore, we may assume without loss of generality that some vertex has
	out-degree one.

	Fix a Hamiltonian cycle and label it as
	\[
		Q=v_0v_1\cdots v_{n-1}v_0
	\]
	in such a way that $d^+(v_0)=1$. Necessarily,
	\[
		N^+(v_0)=\{v_1\}.
	\]
	All subscripts in the rest of the proof are taken modulo $n$. The key point
	is that this local degree restriction propagates backwards around $Q$. We
	make this precise in the following two claims.

	\begin{claim}\label{clm:n-minus-one-indegree}
		If $N^+(v_i)=\{v_{i+1}\}$, then
		$N^-(v_i)=\{v_{i-1}\}$.
	\end{claim}

	\begin{proof}
		By cyclically relabeling the vertices, it is enough to consider $i=0$.
		The arc $v_{n-2}\to v_0$ cannot exist, since otherwise
		\[
			v_0v_1\cdots v_{n-2}v_0
		\]
		would be a cycle of length $n-1$. It follows that
		\[
			d(v_0)=d^+(v_0)+d^-(v_0)
			\leq 1+(n-2)=n-1.
		\]
		Consequently, the hypothesis of the theorem ensures that $v_0$ belongs
		to no bad pair.

		Suppose, to the contrary, that the claim is false, then there exists $v_j\to v_0$ for some
		$1\leq j\leq n-3$. Since $v_j\to v_{j+1}$ is an arc of $Q$, the
		vertices $v_0$ and $v_{j+1}$ have the common in-neighbor $v_j$. They
		must be adjacent; otherwise they form a dominated nonadjacent pair
		containing $v_0$, which is impossible by the preceding paragraph.
		Moreover, $v_0$ has no out-neighbor other than $v_1$, and
		$j+1\geq2$. Hence
		\[
			v_{j+1}\to v_0.
		\]
		Repeating this argument along $Q$ eventually gives
		$v_{n-2}\to v_0$, a contradiction. Thus $v_0$ has no in-neighbor
		other than its predecessor $v_{n-1}$ on $Q$. This proves the claim.
	\end{proof}

	\begin{claim}\label{clm:n-minus-one-propagation}
		If
		\[
			N^+(v_i)=\{v_{i+1}\}
			\quad\text{and}\quad
			N^-(v_i)=\{v_{i-1}\},
		\]
		then $N^+(v_{i-1})=\{v_i\}$.
	\end{claim}

	\begin{proof}
		Again, cyclically relabel the vertices so that $i=0$. Suppose that
		$v_{n-1}\to v_j$ for some $j\neq0$. If $j=1$, then
		\[
			v_1v_2\cdots v_{n-1}v_1
		\]
		is a cycle of length $n-1$, a contradiction. Thus, apart from the
		cycle arc $v_{n-1}\to v_0$, any additional out-neighbor of
		$v_{n-1}$ would have to be some $v_j$ with $2\leq j\leq n-2$.

		For such a $j$, the vertices $v_0$ and $v_j$ are nonadjacent, because
		$N^+(v_0)=\{v_1\}$ and $N^-(v_0)=\{v_{n-1}\}$. Moreover, they have
		the common in-neighbor $v_{n-1}$. Hence $\{v_0,v_j\}$ is a
		dominated nonadjacent pair, and the hypothesis would imply
		$d(v_0)\geq n$. This is impossible because $d(v_0)=2$. Therefore
		$v_{n-1}$ has no out-neighbor other than $v_0$, proving the claim.
	\end{proof}

	We can now complete the propagation. From
	$N^+(v_0)=\{v_1\}$ and Claim~\ref{clm:n-minus-one-indegree}, we obtain
	$N^-(v_0)=\{v_{n-1}\}$. Claim~\ref{clm:n-minus-one-propagation} then gives
	$N^+(v_{n-1})=\{v_0\}$, and another application of
	Claim~\ref{clm:n-minus-one-indegree} gives
	$N^-(v_{n-1})=\{v_{n-2}\}$. Continuing in this way around the entire
	Hamiltonian cycle yields, for every $i$,
	\[
		N^+(v_i)=\{v_{i+1}\}
		\quad\text{and}\quad
		N^-(v_i)=\{v_{i-1}\}.
	\]
	Thus $D$ contains no arcs other than those of $Q$, and consequently
	$D=C_n$.

	The two cases exhaust all possibilities, and the proof is complete.
\end{proof}

\section{Proof of Theorem~\ref{thm:main}}\label{sec:main-proof}

Before we proceed to the proof, we recall a pancyclic result of
H\"{a}ggkvist and Thomassen~\cite{HT1976} that will be used repeatedly.
Their result concerns a vertex outside a directed cycle and turns a 
sufficiently large degree into cycles of all intermediate lengths

\begin{lemma}[H\"{a}ggkvist--Thomassen~\cite{HT1976}]\label{lem:outside-heavy}
	Suppose that a digraph $D$ has $n$ vertices and contains an $(n-1)$-cycle $Q$. 
	If, for the vertex $v \notin Q$, we have $d(v) \geq n$, then $D$ is pancyclic.
\end{lemma}

Next, we turn to two
propositions that drive the subsequent reduction arguments.

\begin{proposition}\label{prop:nonheavy-segment}
	Let $D$ be a non-pancyclic digraph on $n$ vertcies, satisfying $d(x) \geq n$ and $d(y) \geq n$ for every dominating or dominated nonadjacent pair $\{x, y\}$.
	Suppose that $D$ contains an $(n-1)$-cycle $Q=v_0v_1\dots v_{n-2}v_0$, together with a vertex $u$ outside the cycle, and that $v_i \to u$ and $u \to v_j$, where $j \neq i+1$.
	If $\{v_{i+1}, v_{i+2}, \dots, v_{j-1}\} \to u$ or $u\to\{v_{i+1},v_{i+2},\dots, v_{j-1}\}$, then none of $v_{i+1}, v_{i+2},\dots, v_{j-1}$ is of degree at least $n$.
\end{proposition}

\begin{proof}
	By the assumption of the proposition, $D$ is a non-pancyclic digraph.
	At this point, $D$ contains an $(n-1)$-cycle $Q=v_0v_1\dots v_{n-2}v_0$,
	together with a vertex $u$ outside the cycle, such that $v_i \to u$, $u \to v_j$,
	where $ j\neq i+1$.
	Each of $v_{i+1}, v_{i+2}, \dots, v_{j-1}$ is adjacent to $u$
	via an arc of the same direction. Reversing every arc preserves total degrees, heaviness, pancyclicity, and the hypothesis of the proposition, while exchanging the two alternatives in its conclusion. 
	Hence we may assume that all these vertices are out-neighbors of $u$.
	
	Consider the vertex $v_{i+1}$. We see that $V(D)\setminus\{v_{i+1}\}$ forms an
	$(n-1)$-cycle $Q_{i+1} = v_{i}  u v_{i+2}v_{i+3} \dots v_{i-1}v_{i}$.
	By Lemma~\ref{lem:outside-heavy}, if $d(v_{i+1}) \geq n$, then $D$ is pancyclic, contradicting the assumption. Therefore, $v_{i+1}$ cannot be a heavy vertex.
	If $j=i+2$, this proves the proposition. Hence assume $j\geq i+3$.
	We next show that $v_{i+2}$ is not a heavy vertex either.

	Since $V(D)\setminus\{v_{i+1},v_{i+2}\}$ contains the cycle $Q_{i+2} = v_{i}  u v_{i+3}v_{i+4} \dots v_{i-1}v_{i}$, and $u \to v_{i+1}$ is an arc of $D$, we proceed as follows.
	Because $u \to v_{i+1}$ and $u \to v_{i+3}$, if there is no arc between $v_{i+1}$ and $v_{i+3}$,
	then $\{v_{i+1}, v_{i+3}\}$ is a dominating nonadjacent pair. In that case, $v_{i+1}$ would be a heavy vertex, a contradiction.
	Hence there must be an arc between $v_{i+1}$ and $v_{i+3}$.
	
	\caseheading{Case 1.1.} Suppose that $v_{i+1} \to v_{i+3}$. Since
	$u \to v_{i+1}$, the vertex $v_{i+1}$ can be inserted into the cycle $Q_{i+2}$,
	giving the $(n-1)$-cycle $Q'_{i+2} = v_{i}  u v_{i+1}v_{i+3} \dots v_{i-1}v_{i}$,
	which omits only $v_{i+2}$. If $d(v_{i+2}) \geq n$, Lemma~\ref{lem:outside-heavy} would imply that $D$ is pancyclic, a contradiction. Therefore, $v_{i+2}$ is not a heavy vertex.
	
	\caseheading{Case 1.2.} Suppose that $v_{i+3}\to v_{i+1}$. Since
	$v_{i+3}\to v_{i+4}$, we have
	$v_{i+3}\to\{v_{i+1},v_{i+4}\}$. Let $r$ successively ranges over
	\[
	\{i+3,i+4,\ldots,i-2\},
	\]
	where the indices are read in the cyclic order of $Q$ and are taken
	modulo $n-1$. Whenever $v_r\to v_{i+1}$, the vertices $v_{i+1}$ and
	$v_{r+1}$ must be adjacent; otherwise they form a dominating
	nonadjacent pair with common in-neighbor $v_r$, forcing the non-heavy
	vertex $v_{i+1}$ to be heavy. If $v_{r+1}\to v_{i+1}$, the argument
	continues with $r+1$.
	
	Suppose that there exists $0 \le q \le n-2$ with
	$q\notin\{i,i+1,i+2,i+3\}$ such that
	\[
	\{v_{i+3},\ldots,v_{q-1}\}\to v_{i+1}
	\quad\text{and}\quad v_{i+1}\to v_q.
	\]
	Then $v_{i+1}$ can be inserted into $Q_{i+2}$, yielding the
	$(n-1)$-cycle
	\[
	Q'_{i+2}
	=v_iuv_{i+3}\cdots v_{q-1}v_{i+1}v_q\cdots v_{i-1}v_i,
	\]
	which omits only $v_{i+2}$. Lemma~\ref{lem:outside-heavy} therefore
	implies that $v_{i+2}$ is non-heavy. If no such $v_q$ exists, then
	every vertex of the $(n-1)$-cycle $Q_{i+1}$ other than $v_{i+2}$ sends an arc to
	$v_{i+1}$, while $v_{i+1}\to v_{i+2}$. Hence $D$ contains a cycle of
	every length from $3$ to $n$ through $v_{i+1}$ and is pancyclic, again a contradiction.
	Thus, in every case, we either lead to a contradiction or show that $v_{i+2}$ is not heavy.
	
	We now prove that $v_p$ is not a heavy vertex for every $i+1 \leq p \leq j-1$.
	We use induction on $p$. The vertices $v_{i+1}$ and $v_{i+2}$ have
	already been dealt with. Let $i+3\leq p\leq j-1$, and assume that
	$v_{i+1},\ldots,v_{p-1}$ are non-heavy. Consider
	\[
	Q_p=v_iuv_{p+1}v_{p+2}\cdots v_j\cdots v_{i-1}v_i.
	\]
	For each $t\in\{i+1,\ldots,p-1\}$, the vertices $v_t$ and
	$v_{p+1}$ have the common in-neighbor $u$. Hence they are adjacent,
	since $v_t$ is non-heavy.
	
	If $v_{p+1}$ has an in-neighbor in
	$\{v_{i+1},\ldots,v_{p-1}\}$, let
	\[
	\beta=\max\{t\in\{i+1,\ldots,p-1\}:v_t\to v_{p+1}\}.
	\]
	Thus $i+1\leq\beta\leq p-1$ and
	$v_{p+1}\to\{v_{\beta+1},\ldots,v_{p-1}\}$. The path
	$v_{i+1}\cdots v_\beta$ can be inserted between $u$ and $v_{p+1}$,
	giving the cycle $H$ and the path $R$ as follows:
	\[
	H=v_iuv_{i+1}\cdots v_\beta v_{p+1}v_{p+2}
	\cdots v_j\cdots v_{i-1}v_i, \qquad R=v_{\beta+1}\cdots v_{p-1} .
	\]
	
	If $v_{p+1}$ has no
	in-neighbor in $\{v_{i+1},\ldots,v_{p-1}\}$, put instead
	\[
	H=Q_p,\qquad R=v_{i+1}\cdots v_{p-1}.
	\]
	In both cases,
	\[
	V(H)\mathbin{\dot\cup}V(R)=V(D)\setminus\{v_p\},
	\qquad u\to V(R),
	\]
	and all vertices of $R$ are non-heavy.
	If $R$ is empty, then $H$ is an $(n-1)$-cycle omitting $v_p$, and
	Lemma~\ref{lem:outside-heavy} immediately implies that $v_p$ is
	non-heavy. We may therefore assume that $R$ is nonempty.
	
	We claim that $R$ can be successively inserted into $H$.
	Let $v_s$ be the first vertex of $R$ not yet inserted, and let $K$ be the cycle obtained by inserting into $H$ all vertices preceding $v_s$. Starting with $u\to v_s$, traverse $K$ in its directed order.
	If $xy\in A(K)$ and $x\to v_s$, then $x$ is a common in-neighbor of
	$v_s$ and $y$. Since $v_s$ is non-heavy, the $\BGL(n)$ condition forces
	$v_s$ and $y$ to be adjacent. Thus $v_s\to y$ makes $xy$ an insertion
	partner of $v_s$, whereas $y\to v_s$ allows the traversal to continue.
	Consequently, either $v_s$ is inserted or a complete traversal gives
	$K\to v_s$.
	
	In the latter case, test the blocks $R[v_s,v_t]$ for
	$t=s+1,\ldots,p-1$. For a fixed $t$, start with $u\to v_t$ and traverse
	$K$. At an arc $xy$ with $x\to v_t$, the non-heaviness of $v_t$ again
	forces $v_t$ and $y$ to be adjacent. If $v_t\to y$, then $xy$ is an
	insertion partner of $R[v_s,v_t]$, because $K\to v_s$; if
	$y\to v_t$, the traversal continues. Hence a block is inserted, or the
	complete traversal gives $K\to v_t$. After each insertion, replace $K$
	by the resulting cycle and continue from the first remaining vertex of
	$R$. If this process failed to exhaust $R$, it would leave a nonempty
	suffix
	\[
	R^*=v_s v_{s+1}\cdots v_{p-1}
	\]
	with $V(K)=V(Q_p)\cup \{v_{i+1}, \dots v_{s-1} \}$ and $K\to R^*$.
	
	Such a suffix is impossible. For every $t\in\{s,\ldots,p-1\}$ and
	every $x\in V(K)$,
	\[
	v_t v_{t+1}\cdots v_p K[v_{p+1},x]v_t
	\]
	is a directed cycle, using $v_p\to v_{p+1}$ and $x\to v_t$; here
	$K[v_{p+1},v_{p+1}]$ denotes the one-vertex segment. The length of
	the first segment $v_t v_{t+1}\cdots v_p$ ranges over every integer from $1$ to $|R^*|$, 
	and the length of the segment of $K$ ranges over every integer from $0$ to $|K|-1$. 
	Hence, together with the two arcs $v_pv_{p+1}$ and $xv_t$, 
	the total length of the cycle ranges over every integer from $1+0+2=3$ to 
	\[
	|R^*|+|K|-1+2=n.
	\]
	All these cycles realize every length from $3$ to $n$, contradicting the
	non-pancyclicity of $D$.
	
	Therefore, the process exhausts $R$ and insert all vertices of it
	into $K$, producing an
	$(n-1)$-cycle $Q'_p$ with
	\[
	V(Q'_p)=V(D)\setminus\{v_p\}.
	\]
	Lemma~\ref{lem:outside-heavy} now implies that $v_p$ is non-heavy.
	
	Therefore, by mathematical induction, none of $v_{i+1}, v_{i+2}, \dots, v_{j-1}$ is a heavy vertex,
	and Proposition~\ref{prop:nonheavy-segment} holds.
\end{proof}

\begin{proposition}\label{prop:opposite-heavy}
	Let $D$ be a digraph satisfying $d(x) \geq n$ and $d(y) \geq n$ for every dominating or dominated nonadjacent pair $\{x, y\}$.
	Suppose that $D$ contains an $(n-1)$-cycle, the vertex $u$ outside the cycle is of degree $d(u)<n$, 
	and there is a vertex $w$ on the cycle with $d(w)\ge n$ such that $u \leftrightarrow w$. Then $D$ is pancyclic. 
\end{proposition}

\begin{proof}
	Write the $(n-1)$-cycle as $Q=v_0v_1\dots v_{n-2}v_0$, with all subscripts taken modulo $n-1$, 
	and let $v_i$ be an arbitrary heavy vertex on $Q$ such that $v_i \leftrightarrow u$.
	Since $v_{i-1} \to v_i$ and $v_i \to v_{i+1}$, and since $u$ is not heavy, 
	the vertex $u$ must be adjacent to both $v_{i-1}$ and $v_{i+1}$
	to avoid being in a bad pair. We consider the following four cases
	regarding the direction of the arcs between $u$ and $\{v_{i-1}, v_{i+1}\}$.
	
	\caseheading{Case 1.1.} Suppose that $v_{i-1} \to u$ and $u \to v_{i+1}$. 
	Then the $(n-1)$-cycle $Q'_i = v_{i-1} u v_{i+1} \dots v_{i-2} v_{i-1}$ avoids $v_i$.
	Since $d(v_i) \geq n$, Lemma~\ref{lem:outside-heavy} shows that $D$ is pancyclic.
	
	\caseheading{Case 1.2.} Suppose that $v_{i-1} \to u$ and $u \nrightarrow v_{i+1}$, 
	and therefore $v_{i+1} \to u$.
	Since $v_{i+1} \to v_{i+2}$ and $u$ is not heavy, $u$ must be adjacent to $v_{i+2}$. 
	And if $v_{i+2} \to u$, the adjacency of $u$ can be extended to $v_{i+3}$. 
	Similarly, we can successively extend the adjacency of $u$ to the right of $v_{i+1}$, 
	until we encounter $v_{\beta}$, the first vertex for which $v_{\beta} \nrightarrow u$,
	and thus $u \to v_{\beta}$ exists.
	Then the vertices from $v_{i-1}$ to $v_{\beta}$, together with $u$, 
	necessarily form the configuration described in Proposition~\ref{prop:nonheavy-segment}. 
	Now assume that $D$ is non-pancyclic, by Proposition~\ref{prop:nonheavy-segment} 
	$v_i$ can not be heavy, a contradiction. Therefore $D$ is pancyclic.
	
	\caseheading{Case 1.3.} Suppose that $v_{i-1} \nrightarrow u$ and $u \to v_{i+1}$, and therefore $u\to v_{i-1}$.
	Since $v_{i-2} \to v_{i-1}$ and $u$ is not heavy, $u$ must be adjacent to $v_{i-2}$. 
	And if $u\to v_{i-2}$, the adjacency of $u$ can be extended to $v_{i-3}$.
	Similarly, we can continuously extend the adjacency of $u$ to the left of $v_{i-1}$, 
	until we meet $v_{\alpha}$, the first vertex for which $u \nrightarrow v_{\alpha}$,
	and thus $v_{\alpha} \to u$.
	The vertices from $v_{\alpha}$ to $v_{i+1}$, together with $u$,
	then necessarily form the configuration in Proposition~\ref{prop:nonheavy-segment}.
	Now assume that $D$ is non-pancyclic, by Proposition~\ref{prop:nonheavy-segment} $v_i$ is not heavy, a contradiction.
	Therefore $D$ is pancyclic.
	
	\caseheading{Case 1.4.} It remains to consider $v_{i-1} \nrightarrow u$ and $u \nrightarrow v_{i+1}$.
	In this case $u \to v_{i-1}$ and $v_{i+1} \to u$.
	By arguments analogous to Case 1.3 and Case 1.4, the adjacency with $u$ propagates
	in both directions along $Q$.
	Following the propagation backward from $u\to v_{i-1}$ we eventually reach a vertex 
	$v_\alpha$ such that $v_\alpha \to u$,
	and carrying the propagation forward from $v_{i+1} \to u$ we 
	eventually reach a vertex $v_\beta$ such that $v_\beta \to u$,
	and 
	\[
	v_\alpha\to u\to\{v_{\alpha+1},\ldots,v_{i-1}, v_i\},
	\qquad
	\{v_i, v_{i+1}, \ldots,v_{\beta-1}\}\to u\to v_\beta.
	\]
	Put
	\[
	L=v_{\alpha+1}v_{\alpha+2}\cdots v_{i-1},
	\qquad
	R=v_{i+1}v_{i+2}\cdots v_{\beta-1}.
	\]
	Since $u \to v_{i-1}$ and $v_{i+1} \to u$, $L$ and $R$ are nonempty.
	
	Proposition~\ref{prop:nonheavy-segment}, applied to the two segments
	$Q[v_\alpha,v_i]$ and $Q[v_i,v_\beta]$, shows that every vertex of
	$L\cup R$ is non-heavy. We start from the cycle
	\[
	K_0=v_\alpha u v_\beta v_{\beta+1}\cdots
	v_{\alpha-1}v_\alpha,
	\]
	which omits exactly $L$, $v_i$, and $R$, and insert
	all vertices of $L$ and $R$ into $K_0$.
	
	First scan $L$ from left to right and insert its vertecies
	into $K_0$ sequentially. Let $v_s$ be its first uninserted
	vertex and let $K$ be the current cycle. Starting with $u\to v_s$,
	traverse $K$. On an arc $xy$ of $K$ with $x\to v_s$, the non-heavy vertex
	$v_s$ must be adjacent to $y$, since $x$ is their common in-neighbor.
	Thus either $v_s\to y$, making $xy$ an insertion partner of $v_s$, or
	$y\to v_s$ and the scan continues. If no partner occurs, then
	$K\to v_s$. We may then advance the scan one step to the right to $v_{s+1}$.
	On an arc $xy$ with $x\to v_{s+1}$, their common in-neighbor $x$ forces
	$v_{s+1}$ and $y$ to be adjacent. Thus $v_{s+1}\to y$ inserts
	$v_s v_{s+1}$, whereas $y\to v_{s+1}$ continues
	the scan; a complete insertion-failed scan gives $K\to v_{s+1}$,
	and we proceed to the next vertex. 
	Repeating the scan either places all of $L$ into the cycle
	or leaves a suffix
	\[
	L^*=v_s v_{s+1}\cdots v_{i-1}
	\]
	with $K\to L^*$.
	
	Now scan $R$ from right to left and insert its verticies into $K$ sequentially.
	For its last uninserted vertex $v_t$,
	start with $v_t\to u$ and traverse $K$ backwards. On an arc $xy$
	with $v_t\to y$, the vertices $x$ and $v_t$ have the common
	out-neighbor $y$ and must be adjacent. Hence $x\to v_t$ gives an
	insertion partner $xy$ for $v_t$, while $v_t\to x$ continues the scan. An insertion-failed scan
	gives $v_t\to K$. In the latter case move one step to the left to $v_{t-1}$ and scanning
	backwards from $v_{t-1}\to u$, the common out-neighbor $y$ at each arc
	$xy$ forces $x$ and $v_{t-1}$ to be adjacent. 
	Thus any $x\to v_{t-1}$ gives an insertion partner $xy$
	for $v_{t-1}v_t$ because $v_t\to y$,
	while a complete insertion-failed scan proves $v_{s-1}\to K$,
	and we proceed to the next vertex.
	After each insertion we restart with the last remaining vertex,
	which is still denoted by $v_t$.
	Consequently, either all of $R$ is inserted or a prefix
	\[
	R^*=v_{i+1}v_{i+2}\cdots v_t
	\]
	remains with $R^*\to K$.
	
	We now prove that if such a prefix $R^*$ remains then $D$ is pancyclic, thus the conclusion holds. 
	
	If $L$ has already been inserted, then
	\[
	u v_i v_{i+1}\cdots v_q K[y,u],
	\qquad i+1\leq q\leq t,\quad y\in V(K),
	\]
	gives every cycle length from $3$ to $n$ as shown below.
	The length of the two variable parts $ v_i v_{i+1}\dots v_q$ and $K[y,u]$ can be
	$1,\ldots,|R^*|-1$ and $0,\ldots,|K|-1$, respectively.
	Hence, together with the two arcs $uv_i$ and $v_qy$,
	the sequence gives cycles cycles of length from $1+0+2=3$ to $(|R^*|-1)+(|K|-1)+2=n$.
	So $D$ is pancyclic.
	
	If a suffix $L^*$ also remains, then
	\[
	uQ[a,b]K[y,u],
	\qquad a\in V(L^*),\quad b\in V(R^*),\quad y\in V(K),
	\]
	gives all lengths from $4$ to $n$ as explained below, while
	$u v_{i-1}v_i u$ gives length $3$. 
	The length of the variable part $Q[a,b]$ ranges over
	every integer from $2$ to $|L^*|+|R^*|$,
	and the length of $K[y,u]$ ranges over every integer
	from $0$ to $|K|-1$, together with the arcs $ua$ and $by$
    they implement every length from $2+0+2=4$ to 
	$|K|-1+|L^*|+|R^*|+2=|K|+|L^*|+|R^*|+1=n$. Hence
	$D$ is again pancyclic.
	
	Now we may assume that no such prefix $R^*$ remains.
	If a suffix $L^*$ of $L$ remains, run the forward scan once more on the enlarged
	cycle $K$, which now contains $R$. Similarly, an insertion-failure scan would leave
	$L^{**}=v_s\cdots v_{i-1}$ with $K\to L^{**}$, and the cycles
	\[
	v_qv_{q+1}\cdots v_iK[v_{i+1},x]v_q,
	\qquad s\leq q\leq i-1,\quad x\in V(K),
	\]
	would have every length from $3$ to $n$. Indeed, the two variable
	parts $v_{q}v_{q+1}\dots v_i$ and $K[v_{i+1},x]$ have lengths $1,\ldots,|L^{**}|$ and
	$0,\ldots,|K|-1$. 
	Together with the arcs $v_i v_{i+1}$ and $xv_q$,
	the cycles implements every length from $1+0+2=3$ to 
	$|L^{**}|+|K|-1+2=|L^{**}|+|K|+1=n$. 
	Thus $D$ is pancyclic as required.
	
	Finally we can assume a successful insertion which insert the vertex
	set of the two paths $L$ and $R$ into $K_0$, and gives an $(n-1)$-cycle $Q'_i$ with
	\[
	V(Q'_i)=V(D)\setminus\{v_i\}.
	\]
	Since $v_i$ is heavy, Lemma~\ref{lem:outside-heavy} now shows that $D$
	is pancyclic.
	Thus Proposition~\ref{prop:opposite-heavy} holds.
\end{proof}

Now we are ready for the proof of the main theorem.
\begin{proof} [Proof of Theorem \ref{thm:main}]
		The preparatory results in Section~\ref{sec:n-1} now apply. 
		By Theorem~\ref{thm:n-minus-one}, a digraph $D$ satisfying 
		the $\BGL(n)$ condition either contains an $(n-1)$-cycle,
		belongs to $\mathcal{D}_B$, or isomorphic to $C_n$, which is a member of $\mathcal{D}_L$. 
		In the latter two cases, $D$ itself is exceptional, thus the conclusion of Theorem~\ref{thm:main} is immediate. 
		We therefore assume from now on that $D$ contains an $(n-1)$-cycle.
		
		Let $D_0=D$. Choose an $(n-1)$-cycle in $D_0$, denote it by $Q_{n-1}$, and let $u_0$ be the unique vertex outside this cycle.
		
		If $u_0$ is a heavy vertex of $D_0$, then Lemma~\ref{lem:outside-heavy} shows that $D_0$ is pancyclic, and hence Theorem~\ref{thm:main} follows.
		
		If $u_0$ is not heavy, but $u_0$ is joined by a pair of opposite arcs to some heavy vertex on $Q_{n-1}$, 
		then Proposition~\ref{prop:opposite-heavy} again implies that $D_0$ is pancyclic.
		Thus Theorem~\ref{thm:main} also follows in this case.
		
		It remains to consider the following situation: $u_0$ is not heavy, 
		and there is no pair of opposite arcs between $u_0$ and any heavy vertex on $Q_{n-1}$. 
		We show that the subdigraph $D_0-u_0$ satisfies the $\BGL(n-1)$ condition. 
		
		Take any $w\in V(D_0-u_0)$ which belongs to a bad pair in $D_0-u_0$. 
		Note that deleting $u_0$ creates no new bad pairs.
		Hence $w$ belongs to the same bad pair in $D_0$.
		Since $D_0$ satisfies the $\BGL(n)$ condition, $w$ must be heavy in $D_0$, that is,
		\[
		d_{D_0}(w)\geq |D_0|=n .
		\]
		By the assumption of this case, no heavy vertex in $V(D_0-u_0)$ is joined to $u_0$ by a pair of opposite arcs. 
		Therefore, deleting $u_0$ removes at most one arc incident with $w$. Consequently,
		\[
		d_{D_0-u_0}(w)\geq d_{D_0}(w)-1\geq n-1=|D_0-u_0| .
		\]
		Hence $w$ remains a heavy vertex of $D_0-u_0$,
		which means that any vertex belongs to a bad pair in $D_0-v_0$
		is heavy in $D_0-v_0$.
		So $D_0-u_0$ satisfies the $\BGL(n-1)$ condition. Set
		\[
		D_1=D_0-u_0.
		\]
		We now construct recursively a sequence of digraphs
		\[
		D_0,D_1,\ldots,D_p, p \leq n-3
		\]
		such that, 
		\begin{enumerate}
		\item for every $D_i$, $0 \leq i \leq p$, we have
		\[
		|D_i|=n-i,
		\qquad
		D_i\text{ satisfies the }\BGL(n-i)\text{ condition};
		\]
		\item for $0 \leq i \leq p-1$, the digraph $D_i$ contains a cycle of length $n-i-1$, which is denoted by $Q_{n-i-1}$ 
		and the unique vertex in $V(D_i)\setminus V(Q_{n-i-1})$ is denoted by $u_i$, we have
		\[
		D_{i+1}=D_i-u_i;
		\]
		\item $u_i$ is not a heavy vertex of $D_i$, and no heavy vertex on $Q_{n-i-1}$ is joined to $u_i$ by a pair of opposite arcs.
	\end{enumerate}
		
		The digraphs $D_0$ and $D_1$ have already been constructed. Suppose that $D_i$ has been defined. Since $|D_i|=n-i$ and $D_i$ satisfies the $\BGL(n-i)$ condition, Theorem~\ref{thm:n-minus-one} says that $D_i$ either contains an $(n-i-1)$-cycle $Q_{n-i-1}$, 
		or belongs to an exceptional class of the corresponding order. In the first case, we continue by examining the vertex $u_i$ outside the cycle.
		
		If $u_i$ is a heavy vertex of $D_i$, then Lemma~\ref{lem:outside-heavy} implies that $D_i$ is pancyclic, and the sequence stops at $D_i$.
		
		If $u_i$ is not heavy, but it is joined by a pair of opposite arcs to some heavy vertex of $D_i-u_i$, then Proposition~\ref{prop:opposite-heavy} shows that $D_i$ is also pancyclic, and again the sequence stops at $D_i$.
		
		Suppose now that $u_i$ is not heavy and that no heavy vertex of $D_i-u_i$ is joined to $u_i$ by a pair of opposite arcs. 
		We show that the induced subdigraph $D_i-u_i$ satisfies the $\BGL(n-i-1)$ condition. 
		
		Take any $w\in V(D_i-u_i)$ which belongs to a bad pair in $D_i-u_i$. 
		Note that deleting $u_i$ creates no new bad pairs, so $w$ belongs to
		the same bad pair in $D_i$.
		And since $D_i$ satisfies the $\BGL(n-i)$ condition, $w$ must be heavy in $D_i$,
		that is
		\[
		d_{D_i}(w)\geq |D_i|=n-i .
		\]
		
		After $u_i$ is deleted, the absence of a pair of opposite arcs between $u_i$ 
		and any heavy vertex means that $w$ loses at most one in-arc or out-arc.
		Hence
		\[
		d_{D_i-u_i}(w)\geq d_{D_i}(w)-1\geq n-i-1 =|D_i-u_i|.
		\]
		Thus $w$ remains a heavy vertex of $D_i-u_i$. 
		Since $w$ is an arbitrarily chosen vertex belonging to a bad pair in $D_i-u_i$,
		$D_i-u_i$ satisfies the $\BGL(n-i-1)$ condition. We may thus set
		\[
		D_{i+1}=D_i-u_i,
		\]
		and continue the recursive construction.
		
	According to the preceding discussion, the construction either stops at a pancyclic digraph $D_p$ with $p \leq n-3$, terminates at an exceptional digraph $D_p$,
	or continues all the way to $p=n-3$. If some $D_p$ is pancyclic, then it contains cycles of every length from $3$ to $n-p$. During the construction we have, in turn, found cycles of lengths $n-1,\ldots,n-p+1$, while the existence of an $n$-cycle is guaranteed by Theorem~\ref{thm:BGL-hamiltonian1}. Hence $D=D_0$ is pancyclic. If the sequence continues to $p=n-3$, the same reasoning applies: the construction provides cycles of lengths $n-1,\ldots,3$, and Theorem~\ref{thm:BGL-hamiltonian1} supplies an $n$-cycle. Thus $D_0$ is pancyclic.
	
	We are left with the case in which $D_p$ is an exceptional digraph. We divide the argument into two cases.
		
		\caseheading{Case 1. $D_p\in\mathcal{D}_B$.}
		
		If $p=0$, then $D_0\in\mathcal{D}_B$, and the conclusion of Theorem~\ref{thm:main} is immediate.
		Assume from now on that $p\geq1$. Since the sequence continues to at most $p=n-3$, $|D_p|\ge 3$.
		As $D_p\in\mathcal{D}_B$ is balance bipartite, $|D_p|$ must be even and thus $|D_p|\ge 4$.
		Let $A$ and $B$ be the two partite sets of $D_p$, where
		\[
		|A|=|B|=a\geq2,
		\]
		and write
		\[
		x=u_{p-1}.
		\]
		Choose any $v\in V(D_p)$. Since $a\geq2$, we may choose another vertex $w\neq v$ in the same partite set as $v$. Then $\{v,w\}$ is a bad pair in $D_{p-1}$. Moreover,
		\[
		d_{D_p}(v)=2a,
		\qquad
		|D_{p-1}|=2a+1.
		\]
		Because $D_{p-1}$ satisfies the $\BGL(2a+1)$ condition, $d_{D_p}(v)=2a+1$, and $v$ has be adjacent to $x$.
		Since $v$ is arbitrarily chosen, $x$ is adjacent to every vertex of $D_p$.
		
		Let
		\[
		O=N_{D_{p-1}}^+(x),
		\qquad
		I=N_{D_{p-1}}^-(x).
		\]
		Since $D_{p-1}$ is strongly connected and $x$ is adjacent to every vertex of $D_p$, we have
		\[
		O\neq\varnothing,
		\qquad
		I\neq\varnothing,
		\qquad
		O\cup I=V(D_p).
		\]
		It can be verified that we can choose vertices from $O$ and $I$ that lie in different partite sets. 
		After interchanging the names of $A$ and $B$ if necessary, we may assume that
		\[
		x^*\in O\cap A,
		\qquad
		y^*\in I\cap B.
		\]
		For each $1\leq r\leq a$, choose distinct vertices in $A$ and $B$ as follows:
		\[
		x_0=a^*,x_1,\ldots,x_{r-1}\in A,
		\qquad
		y_0,\ldots,y_{r-1}=y^*\in B.
		\]
		Since $D_p$ is a balanced complete bipartite digraph,
		\[
		x x_0 y_0 \cdots x_{r-1} y_{r-1} x
		\]
		is a cycle of length $2r+1$. Thus $D_{p-1}$ contains odd cycles of every length in
		\[
		3,5,\ldots,2a+1.
		\]
		On the other hand, $D_p$ itself contains even cycles of every length in
		\[
		4,6,\ldots,2a.
		\]
		Hence $D_{p-1}$ is pancyclic.
		
		Combining these cycles with the $(|D_i|-1)$-cycles chosen at the earlier recursive steps gives cycles of every length $3,4,\ldots,n-1$ in $D_0$. Finally, Theorem~\ref{thm:BGL-hamiltonian1} gives a cycle of length $n$ in $D_0$. Therefore, $D_0$ is pancyclic.
		
		\caseheading{Case 2. $D_p=C_{n-p}$.}
		
		We now prove that, when the deleted vertices are added back one at a time,
		each resulting digraph remains locally semicomplete,
		and it follows the original digraph $D_0$ is locally semicomplete. 
		It is equivalent to prove that adding back the deleted vertices produces no bad pair. 
		Since the recursion reaches $D_p$, for every $0\leq i<p$, the vertex $u_i$ is not heavy in $D_i$, and $D_i$ satisfies the $\BGL(n-i)$ condition.
		
		First, $D_p=C_{n-p}$ contains no bad pair. 
		
		Starting from $D_p$, add back the vertices
		\[
		u_{p-1},u_{p-2},\ldots,u_0
		\]
		one by one, together with the necessary arcs to reconstruct
		\[D_{p-1}, D_{p-2} \ldots, D_0.\]
		We use reverse induction on $i=p-1,p-2,\ldots,0$ to prove that $D_i$ contains no bad pair. Assume that $D_{i+1}$ contains no bad pair, and set
		\[
		F=D_{i+1},
		\qquad
		G=D_i=F+u_i.
		\]
		Let $|F|=m$, so $|G|=m+1$. Suppose, to the contrary, that $G$ contains a bad pair $\{x,y\}$.
		
		If $u_i\in\{x,y\}$, then, because $G=D_i$ satisfies the $\BGL(m+1)$ condition,
		\[
		d_G(u_i)\geq |D|=m+1.
		\]
		Thus $u_i$ is a heavy vertex of $D_i$, contradicting the condition imposed on this recursive branch.
		
		Therefore, $x,y\in V(F)$. Adding $u_i$ does not alter adjacency between old vertices, so $x$ and $y$ remain nonadjacent in $F$. Since $F$ contains no bad pair, $x$ and $y$ have neither a common out-neighbor nor a common in-neighbor in $F$; that is,
		\[
		N_F^+(x)\cap N_F^+(y)=\varnothing,
		\qquad
		N_F^-(x)\cap N_F^-(y)=\varnothing.
		\]
		Take any $z\in V(F)\setminus\{x,y\}$. Among the four possible arcs
		\[
		x\to z,
		\qquad
		y\to z,
		\qquad
		z\to x,
		\qquad
		z\to y,
		\]
		at most two can occur. Otherwise, if both $x\to z$ and $y\to z$ occur, then $z$ is a common out-neighbor of $x$ and $y$; if both $z\to x$ and $z\to y$ occur, then $z$ is a common in-neighbor of $x$ and $y$. Either possibility is a contradiction. Hence each $z\in V(F)\setminus\{x,y\}$ contributes at most $2$ to $d_F(x)+d_F(y)$. Since $x$ and $y$ are nonadjacent, it follows that
		\[
		d_F(x)+d_F(y)\leq 2(m-2).
		\]
		After $u_i$ is added back, there are at most two arcs between $u_i$ and $x$, and at most two arcs between $u_i$ and $y$. Therefore,
		\[
		\begin{aligned}
		d_G(x)+d_G(y)
		&\leq d_F(x)+d_F(y)+4\\
		&\leq 2(m-2)+4
		 =2m.
		\end{aligned}
		\]
		On the other hand, $\{x,y\}$ is a bad pair in $G$, and $G$ satisfies the $\BGL(m+1)$ condition. Hence
		\[
		d_G(x)\geq m+1,
		\qquad
		d_G(y)\geq m+1.
		\]
		Consequently,
		\[
		d_G(x)+d_G(y)\geq 2m+2,
		\]
		contradicting the preceding upper bound. Thus $G=D_i$ contains no bad pair. Reverse induction now shows that every $D_i$ contains no bad pair; in particular, $D_0$ contains no bad pairs and hence it is locally semicomplete.
		
		Therefore, in this case, if $D_0$ is not pancyclic, then it belongs to the exceptional class of non-pancyclic $\LSD$s,
		namely $\mathcal{D}_L$, allowed by the condition of the theorem.
		
		Combining the two cases, Theorem~\ref{thm:main} follows.
	\end{proof}
\section{Two Open Problems} \label{sec:OpenProblem}
While Theorem \ref{thm:BGL-hamiltonian1} restricts only one direction of nonadjacent pairs, 
our condition applies to both dominating and dominated nonadjacent pairs, with a slightly higher lower bound.
Based on this, we raise the following two problems toward extending Theorem \ref{thm:BGL-hamiltonian1} to pancyclicity.

As argued in the introduction, the condition of Theorem \ref{thm:BGL-hamiltonian1} covers 
the class of LISD. Therefore, in order to extend it to pancyclicity, 
the non-pancyclic LISDs must be necessarily included in the exceptional class.
A complete characterization of these digraphs thus appears to be an
essential step in this direction.

\begin{problem}
	Characterize the class of non-pancyclic LISDs.
\end{problem}

On the other hand, it might be a worthwhile attempt to restore the lower bound to the level of Theorem \ref{thm:BGL-hamiltonian1} 
and see whether pancyclicity is still guaranteed. 
In this setting, a third class of exceptional digraphs arises. 
Let $\mathcal{D}^-_{B}$ denote the class of digraphs obtained by removing one arc 
from a digraph in $\mathcal{D}_B$.
\begin{problem}
	Let $D$ be a digraph of order $n$ such that for every dominated nonadjacent pair 
	or dominating nonadjacent pair $\{u,v\}$, either
	 $$d(u)\geq n, \qquad d(v)\ge n-1,$$ 
	 or
	 $$d(u)\geq n-1, \qquad d(v)\ge n.$$
	Suppose further that $D$ does not belong to $\mathcal{D}_L$,
	$\mathcal{D}_B$ or $\mathcal{D}^-_B$. Is $D$ pancyclic?
\end{problem}


\begin{thebibliography}{99}
	
\bibitem{JBJ1990}
Bang-Jensen. J. (1990). Locally semicomplete digraphs: a generalization of
 tournaments. J. Graph Theory, 14(3):371--390.

\bibitem{BG2007}
Bang-Jensen, J., \& Gutin, G. (2007). 
Theory, algorithms and applications. Springer Monographs in Mathematics, Springer-Verlag London Ltd., London, 101.

\bibitem{BGL1996}
Bang-Jensen, J., Gutin, G., \& Li, H. (1996). Sufficient conditions for a digraph to be
Hamiltonian. J. Graph Theory, 22(2), 181–187.

\bibitem{BHP1993}
Bang-Jensen, J., Huang, J., and Prisner, E. (1993). In-tournament digraphs. J. Combin. Theory Ser. B, 59(2):267--287.
	
\bibitem{Bondy1973}
Bondy, J. A. (1973). Pancyclic graphs: recent results, Infinite and finite sets (Colloq., Keszthely, 1973; 
dedicated to P. Erdos on his 60th birthday), Vol. I, 181--187. In Colloq. Math. Soc. J\'{a}nos Bolyai (Vol. 10).

\bibitem{DK2017}
Darbinyan, S. K., \& Karapetyan, I. A. (2017).
On longest non-Hamiltonian cycles in digraphs with the conditions of
Bang-Jensen, Gutin and Li.
Discrete Applied Mathematics, 216, 537--549.

\bibitem{Fan1984}
Fan, G. H. (1984). New sufficient conditions for cycles in graphs. Journal of Combinatorial Theory, Series B, 37(3), 221--227.

\bibitem{Gho1960}
Ghouila-Houri, A. (1960). Une condition suffisante d'existence d'un circuit hamiltonien. C.R. Acad. Sci. Paris, 25, 495--497.

\bibitem{Guo1995}
Guo, Y. (1995). Locally semicomplete digraphs. Verlag der Augustinus-Buchhandlung.

\bibitem{HT1976}
H\"{a}ggkvist, R., \& Thomassen, C. (1976). On pancyclic digraphs.
Journal of Combinatorial Theory, Series B, 20(1), 20--40.
\href{https://doi.org/10.1016/0095-8956(76)90063-0}{doi:10.1016/0095-8956(76)90063-0}.

\bibitem{Mey1973}
Meyniel, H. (1973). Une condition suffisante d'existence d'un circuit hamiltonien dansun graphe orient'e. 
J. Combin. Theory Ser. B, 14, 137–-147.

\bibitem{Woo1972}
Woodall D.R. (1972). Sufficient conditions for cycles in digraphs. Proc. London Math. Soc., 24, 739–-755.

\bibitem{ZWHL2026}
Zhang, Z. B., Wu, W., He, W., \& Li, T. (2026). 
Cycles of lengths $3$ and $n-1$ in digraphs under a Bang-Jensen–Gutin–Li type condition. Graphs and Combinatorics, 42(2), 25.


\end{thebibliography}
\end{document}